\documentclass[reqno,english]{amsart}
\usepackage[utf8]{inputenc}
\usepackage{amsfonts, amssymb, amsmath, amsthm}
\usepackage{mathtools}
\usepackage{enumerate}
\usepackage{graphicx}
\usepackage{fullpage}
\usepackage{verbatim}
\usepackage{hyperref}
\usepackage{tikz-cd}
\usepackage{todonotes}
\hypersetup{colorlinks=true, linkcolor=red, urlcolor=blue, citecolor=blue, pdfpagemode=UseNone, pdfstartview=FitH, bookmarksopen=true}
\hypersetup{pdftitle=Title}

\newtheorem*{theorem*}{Theorem}
\newtheorem{theorem}{Theorem}

\newtheorem*{lemma*}{Lemma}
\newtheorem{lemma}[theorem]{Lemma}
\newtheorem{problem}[theorem]{Problem}
\newtheorem*{proposition*}{Proposition}
\newtheorem{proposition}[theorem]{Proposition}
\newtheorem*{fact*}{Fact}

\newtheorem*{question*}{Question}

\newtheorem{conjecture}[theorem]{Conjecture}
\newtheorem*{corollary*}{Corollary}
\newtheorem{corollary}[theorem]{Corollary}

\numberwithin{claimcounter}{theorem}
\newtheorem*{claim*}{Claim}

\theoremstyle{remark}
\newtheorem*{remark*}{Remark}
\newtheorem{remark}[theorem]{Remark}
\newtheorem{example}[theorem]{Example}

\theoremstyle{definition}
\newtheorem*{definition*}{Definition}
\newtheorem{definition}[theorem]{Definition}

\numberwithin{theorem}{section}

\newtheorem*{observation*}{Observation}

\newcommand{\norm}[1]{\left\|#1\right\|} % a norm

\newcommand{\N}{\mathbb{N}}

\newcommand{\F}{\mathbb{F}}
\renewcommand{\P}{\mathbb{P}}

\DeclareMathOperator{\lex}{lex}
\DeclareMathOperator{\st}{st}

\DeclareMathOperator{\interior}{int}

\DeclareMathOperator{\sd}{sd}
\DeclareMathOperator{\lk}{lk}

\DeclareMathOperator{\SPAN}{span}

\DeclareMathOperator{\tail}{Tail}

\DeclareMathOperator{\fvec}{\bf f}

\DeclareMathOperator{\rinj}{r_{inj}}
\DeclareMathOperator{\cl}{cl}
\DeclareMathOperator{\rscv}{r_{scv}}
\DeclareMathOperator{\rdel}{r_{Del}}
\DeclareMathOperator{\Del}{Del}

\newcommand{\exshift}[1]{\Delta^{ex}(#1)}

\DeclareMathOperator{\length}{length}
\newcommand{\bbeta}{\boldsymbol{\beta}}

\newif\ifcmnts
\cmntstrue %Comments on
\ifcmnts
\newcommand{\marrow}{\marginpar{\boldmath$\longleftarrow$}}
\newcommand{\denys}[1]{\ifhmode\newline\fi\marrow
  \textsf{\textcolor{magenta}{\bf
      Denys:} #1\newline}}
\else
\newcommand{\denys}[1]{}
\newcommand{\fg}[1]{}
\fi

\title{The typical algebraic shifting of a surface}
\author{Denys Bulavka}
\address{Department of Mathematics \& Statistics\\ Dalhousie University\\ 6297 Castine Way\\ PO BOX 15000\\ Halifax\\ NS\\ Canada\\ B3H 4R2}
\email{denys.bulavka@dal.ca}
\thanks{D. Bulavka was partially supported by AARMS postdoctoral fellowship and by grant ISF-687/24.}
\author{Eran Nevo}
\address{Einstein Institute of Mathematics\\
 Hebrew University\\ Jerusalem~91904\\ Israel.}
\email{nevo@math.huji.ac.il}
\thanks{E. Nevo was partially supported by the Israel Science Foundation grants ISF-2480/20 and ISF-687/24.}
\author{Yuval Peled}
\address{Einstein Institute of Mathematics\\
 Hebrew University\\ Jerusalem~91904\\ Israel.}
\email{yuval.peled@mail.huji.ac.il}
\thanks{Y. Peled was partially supported by the Israel Science Foundation grant ISF-3464/24.}

\begin{document}

\begin{abstract}
We initiate a statistical study of Kalai's exterior algebraic shifting, focusing on concentration phenomena for random triangulations of a fixed space. First, for a uniform $n$-vertex refinement of any given graph $G$, we show that asymptotically almost-surely (a.a.s.) its exterior algebraic shifting is an explicit shifted graph depending only on $n$ and the Betti numbers of $G$. 
Next, for any given compact connected Riemannian surface $S$, sample $n$ points independently at random according to the volume measure, and consider the resulted a.a.s. unique Delaunay triangulation. We prove that a.a.s. its exterior algebraic shifting is an explicit shifted complex depending only on $n$ and the genus of $S$. In both results the expected shifted complex is a \emph{homology lex-segment} complex, a notion we define combinatorially and characterize numerically a l\'{a} Bj\"{o}rner-Kalai. 

As a tool to prove the result on surfaces, we prove a universality result on edge contractions: for every fixed surface triangulation $K$, every dense enough point set in the surface yields a Delaunay triangulation that edge contracts to $K$.
\end{abstract}

\maketitle

\section{Introduction}
A simplicial complex $K$ on the vertex set $[n]$ is \emph{shifted} if for all $1\le i <j$,   $j\in A\in K$ implies that 
$(A\setminus \{j\})\cup \{i\} \in K$.
Erd\H{o}s, Ko and Rado~\cite{erdos61-intersection} introduced
combinatorial shifting operators, which reduce problems on simplicial complexes to the shifted case. These operators have been used 
in many problems in extremal combinatorics, see Frankl's survey~\cite{frankl87-shifting}. 
Kalai \cite{KalaiPart1} introduced the exterior algebraic shifting operator on simplicial complexes, which is a canonical way to associate a shifted complex with a given simplicial complex, based on the exterior face ring over a fixed infinite field. 
Algebraic shifting has found many applications in combinatorics, especially in $f$-vector theory, and is interesting on its own, see, e.g., Kalai's survey 
\cite{kalai02-shifting}.  

While algebraic shifting preserves the face numbers and Betti numbers of the  simplicial complex, it is not determined by them in general, not even in dimension $1$, see Example~\ref{ex:graph} below. However, it is determined for all $n$ vertex triangulations of a fixed 1-dimensional compact manifold, by the properties mentioned above together with Theorem~\ref{t:simplex-union} below in order to combine the exterior algebraic shifting of each connected component.
In dimension 2, 
for any fixed compact connected surface without boundary, all its $n$ vertex triangulations have the same face numbers (and of course also the same Betti numbers).
While the exterior algebraic shifting is constant on $n$ vertex triangulation of the 2-sphere, see~\cite[Sec.2.4]{kalai02-shifting} or \cite[Footnote 2]{keehn24-exterior}, it is not constant for higher genus surfaces; see~\cite{keehn24-exterior} for a characterization of the possible shifted complexes for small genus surfaces. 
Although the exterior algebraic shifting is not constant in these cases, it was observed in~\cite[Remarks 5.3]{keehn24-exterior} that many of the triangulations of the torus have the same exterior algebraic shifting. Naturally the following question arises: is there a \emph{typical} exterior algebraic shifting of a surface triangulation?

We study the expected asymptotic behavior of the exterior algebraic shifting  
over triangulations of a fixed compact space $X$.
Two natural random models come to mind:
(1)
fix the topology of the space and sample a random $n$-vertex triangulation $U_n(X)$ of it
uniformly, or, 
(2)
fix a metric and a volume form on the space
and consider the random Delaunay model, where an $n$ points set $P_n$ is sampled at random according to the volume measure, and the unique Delaunay complex $\Del(P_n)$ they define is considered -- this complex turns out to be an embedded triangulation of $X$, for $X$ a Riemannian surface and $n$ large enough~\cite{leibon99-delaunay,dyer08-surface}.

We analyze the first model for $1$-dimensional spaces and the second model for closed connected Riemannian surfaces. (We will consider only surfaces without boundary.) 
In both cases we show \textit{concentration}, namely, that a.a.s. (i.e., with probability tending to $1$ as $n\to\infty$) the resulted shifted complex is the unique \emph{homology lex-segment} on $n$ vertices with the same homology as $X$, denoted by $\Delta(X,n)$ and defined next. 

Let $G=(V,E)$ be a fixed graph, and $|G|$ denote its geometric realization, namely its corresponding topological space. Denote: $c$ is the number of connected components of $G$, $|V|=m+c-1$ its number of vertices, $|E|=m+b-1$ its number of edges; thus $b=\beta_1(|G|)$, the first Betti number of $G$.
Let $K$ be a subdivision of $G$ on $n+c-1$ vertices. Then we have: the exterior algebraic shifting $\exshift K$ is the same over any field and 
is a homology lex-segment if and only if it equals 
$$\Delta(|G|,n+c-1) := [n+c-1] \bigcup \left \{A \in \binom{[n]}{2}\colon A \leq \{2,b+2\} \right \}.$$
Here and throughout the rest of the article the order $\leq$ stands for the lexicographic order.
Note that $\{2,b+2\}$ is the $n+b-1$th edge in the lex-order on $\binom{[n]}{2}$.

Now, let us denote by $S_g$ the orientable closed connected surface with genus $g$. For a triangulation $K$ of $S_g$ with $n$ vertices, where $n\ge 4+6g$, we have that its exterior algebraic shifting
$\exshift K$, over a field of characteristic zero,
is a homology lex-segment if and only if it equals 
$$\Delta(S_g, n):= [n] \bigcup \left \{\sigma\in \binom{[n]}{2}\colon A \leq \{4,4+6g\} \right \} \bigcup \left \{ A \in \binom{[n]}{3} \colon A \leq \{1,4,4+4g\} \right \} \bigcup \{\{2,3,4\}\}.$$

The exterior algebraic shifting of a non-orientable surface $N_g$ of genus $g$ depends on the characteristic of the field.
Let $K$ be a triangulation of $N_g$ with $n$ vertices, where $n\geq 4+3g$, then its exterior algebraic shifting $\Delta^{ex}_2(K)$ over a field of characteristic $2$ is a homology lex-segment if and only if it equals
$$\Delta_2(N_g, n):= [n] \bigcup \left \{A \in \binom{[n]}{2}\colon A \leq \{4,4+3g\} \right \} \bigcup \left \{ A \in \binom{[n]}{3} \colon A \leq \{1,4,4+2g\} \right \} \bigcup \{\{2,3,4\}\}.$$
On the other hand, its exterior algebraic shifting 
$\Delta_0^{ex}(K)$ over a field of characteristic zero
is a homology lex-segment if and only if it equals
$$\Delta_0(N_g, n):= [n] \bigcup \left \{A\in \binom{[n]}{2}\colon A \leq \{4,4+3g\} \right \} \bigcup \left \{ A\in \binom{[n]}{3} \colon A \leq \{1,4,5+2g\} \right \}.$$
We will however avoid explicitly mentioning the characteristic of the field except where it is necessary.

The following universality result is our main technical result, of independent interest. It says that for \emph{every} fixed surface triangulation $K$, all fine enough Delaunay triangulations of same surface do edge contract to $K$. Formally:
\begin{theorem}[Universality for edge contractions]
    \label{t:delaunayref}
    Let $S$ be a closed connected Riemannian surface and $K$ a triangulation of $S$. Then, there exists $\rho>0$ small enough such that if $P\subseteq S$ is a $\rho$-dense point set, locally in general position, then there exists a sequence of edge contractions from $\Del(P)$ to $K$ such that each intermediate complex triangulates $S$.
\end{theorem}
Being \emph{locally in general position} here means that within every small neighborhood, whose size depends on $S$, no point of $P$ is on a geodesic between other two points of $P$ and no four of its points lie on a circle.
Theorem~\ref{t:delaunayref} is used to prove the following asymptotic behavior of a random Delaunay triangulation, namely concentration for the exterior algebraic shifting.
\begin{theorem}[Concentration of exterior algebraic shifting for Random Delaunay on surfaces]\label{t:aasDelauney}
Let $S$ be a closed connected Riemannian surface. Then, for the random Delaunay model on $S$ there holds:
\begin{itemize}
    \item If $S$ is orientable then, a.a.s., the exterior algebraic shifting satisfies
$\Delta^{ex}(\Del(P_n))
    =\Delta(S,n)
    $.
    \item If $S$ is nonorientable and $p\in \{0,2\}$ then, a.a.s., 
the exterior algebraic shifting over any fixed field of characteristic $p$ satisfies $\Delta^{ex}_p(\Del(P_n)) 
    =\Delta_p(S,n)$,
    where $\Delta_p$ stands for the exterior algebraic shifting operation over a field of characteristic $p$.
\end{itemize}

\end{theorem}
For uniform random triangulations of a one dimensional compact topological space we show the following concentration for the exterior algebraic shifting.
\begin{theorem}[Concentration of exterior algebraic shifting for uniform triangulations in dimension 1]\label{t:aasUniform1-dim}
   Let $G$ be a finite graph. Then, a.a.s., the exterior algebraic shifting over any fixed field satisfies $\exshift {U_n(|G|)}
   = \Delta(|G|,n)$.
\end{theorem}
Note that the probabilistic conclusion in this theorem can not be upgraded to a deterministic one:
\begin{example}\label{ex:graph}
Let $G$ be the graph on 5 vertices containing $K_4$ and 7 edges, denote by $v$ its vertex of degree one, and by $u$ the neighbor of $v$. Let $R$ be the $n$ vertex refinement of $G$ obtained by subdividing the edge $vu$ by $n-5$ new vertices. Then $K_4$ is a subgraph of $R$ hence the edge $\{3,4\}\in \exshift R$. But by Theorem~\ref{t:aasUniform1-dim}  a.a.s. $\{3,4\}\notin \exshift H$ for a uniform $n$ vertex refinement $H$ of $G$.
\end{example}

Theorems~\ref{t:aasUniform1-dim} and~\ref{t:aasDelauney} lead to the following natural conjecture.
\begin{conjecture}\label{conj:uniform_surface}
For every fixed genus $g$, a.a.s., the exterior algebraic shifting over any fixed field satisfies
     $$\exshift {U_n(S_g)}
     =\Delta(S_g,n)\,.$$
\end{conjecture}

{\bf Outline.} 
In Section~\ref{sec:Prelim} we provide preliminaries on exterior algebraic shifting and on Delaunay triangulations; 
in Section~\ref{sec:Universality} we prove Theorem~\ref{t:delaunayref};  in Section~\ref{sec:assDelaunay} we prove Theorem~\ref{t:aasDelauney}; in Section~\ref{sec:ass1dim} we prove Theorem~\ref{t:aasUniform1-dim}; in Section~\ref{sec:Conclude} we discuss concentration for triangulations of other spaces and some open problems, and in Appendix \ref{apx:lex} we discuss homology lex-segment complexes, which may be of independent interest.

\section{Preliminaries}\label{sec:Prelim} 

\subsection{Algebraic shifting}
We recall Kalai's definition~\cite{KalaiPart1}.
Let $K$ be a simplicial complex with vertex set $N$, $\F$ be an field extension of a base field $\F'$ of transcendence degree at least $n^2$, and consider the exterior algebra mod the ideal generated by non-faces $\bigwedge K = \bigwedge \F^{N}/\SPAN \{e_A \colon A \notin K\}$, called the  \emph{exterior face ring of $K$}. Here the $e_i$'s form the standard basis of $\F^{N}$, and $e_A$ stands for the exterior product $e_{a_1}\wedge\ldots\wedge e_{a_k}$ where $A=\{a_1<\ldots<a_k\}$.
Take $(f_i)_{i\in N}$ a generic change of basis of $\F^{N}$, namely all $n^2$ entries of the transition matrix are algebraically independent over the base field $\F'$. Then, the \emph{exterior algebraic shifting of $K$} is the simplicial complex given by $$\exshift K = \{B \subseteq N \colon \overline{f_B} \notin \SPAN_{\F} \{\overline{f_A}\colon A<B\}\},$$
where $f_A=f_{a_1}\wedge\ldots\wedge f_{a_k}$ and its bar $\overline{f_A}$ denotes its image in $\bigwedge K$.

We now recall how to compute the exterior algebraic shifting of a union of a pair of simplicial complexes along a simplex. For it, for a set $T=\{t_1<\cdots <t_r\}\subseteq N$ set $D_K(T) = |\{ t \in N \colon t_r<t,~T\cup \{t\} \in \exshift{K} \}|$ to be the number of single vertex extensions of $T$ having $T$ as the smallest member of their shadow. We will make use of the following result~\cite[Theorem 4.6]{nevo05-basic}.
\begin{theorem}
  \label{t:simplex-union}
  Let $K$ and $L$ be simplicial complexes, with vertex set $(K\cup L)_0 = N$, intersecting along a simplex $K\cap L = \langle B\rangle = \{A \colon A\subseteq B\}$. Let $T=\{t_1<\cdots <t_r\}\subseteq N$, then
  $$T\in \exshift{K\cup L} \iff t_r-t_{r-1}\leq D_K(T\setminus t_r)+D_L(T\setminus t_r)-D_{\langle B \rangle}(T\setminus t_r).$$
  
\end{theorem}

\begin{corollary}
  \label{c:gluehomlex}
  Let $K$ and $K'$ be triangulations of surfaces $S$ and $S'$ resp., intersecting in  
  a single $2$-face $B$, i.e., $K\cap K' = \langle B\rangle$.
  Then, $(K\cup K')\setminus \{B\}$ is a triangulation of the connected-sum surface $(S\cup S) \setminus B$.
  Moreover, if $\exshift K$ and $\exshift{K'}$ are homology lex-segment complexes then $\exshift{(K\cup K') \setminus \{ B \} }$ is a homology lex-segment complex as well.
\end{corollary}
\begin{proof}
    We assume that both $S$ and $S'$ are orientable. The argument for non-orientable surfaces is analogous.
    Let us assume that $K$ and $K'$ have $n$ and $n'$ vertices respectively. The vertex set of $\exshift{K \cup K'}$ is then $[n+n'-3]$. In particular, since $n\geq 4+6g$ and $n'\geq 4+6g'$ we have that $n+n'-3\geq 4+6(g+g')$ as required by the homology lex-segment definition. Combining the expressions $\Delta(S_g,n)$, see also Equation~\eqref{eq:hsegsrf}, with Theorem~\ref{t:simplex-union} we have that 
    \begin{align*}
    \exshift{K\cup K'} &= \{ A \in \binom{[n+n'-3]}{2}\colon A \leq \{4,4+6(g+g')\}\} \\ &\cup \{ A \in \binom{[n+n'-3]}{3}\colon A \leq \{1,4,4+4(g+g')\}\}\\ &\cup \{\{2,3,4\},\{2,3,5\}\}.
    \end{align*}
    Removing $\sigma$ from $K\cup K'$ decreases $\beta_2(K\cup K')$ by one and since the resulting complex is shifted we obtain that $\exshift{(K\cup K') \setminus \{B\}} = \exshift{K\cup K'} \setminus \{2,3,5\}$ which is the desired homology lex-segment complex.
\end{proof}

\begin{figure}[!htb]
\minipage{0.30\textwidth}
\includegraphics[width=\linewidth]{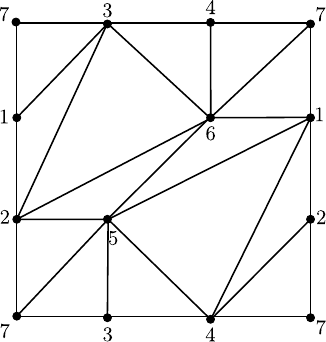}
\centering (1)
\endminipage\hfill
\minipage{0.30\textwidth}
\includegraphics[width=\linewidth]{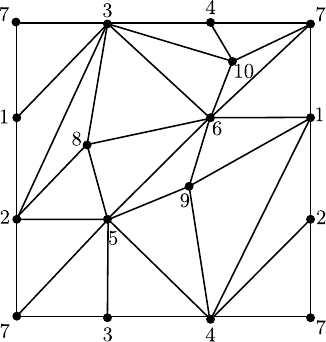}
\centering (2)
\endminipage\hfill
\minipage{0.30\textwidth}%
\includegraphics[width=\linewidth]{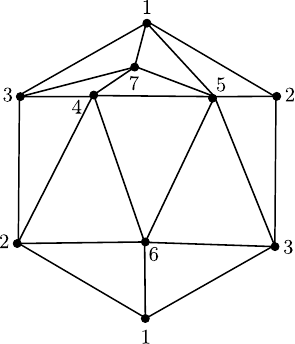}
\centering (3)
\endminipage
\caption{(1) An irreducible triangulation of the torus on $7$ vertices. (2) Triangulation of the torus on $10$ vertices whose exterior algebraic shifting is a homology lex-segment. (3) Triangulation of the projective plane on $7$ whose exterior algebraic shifting is a homology lex-segment over fields with characteristic $0$ and $2$.}
\label{fig:triangulations}
\end{figure}

\begin{corollary}
    \label{c:hlextriang}
    Let $S$ be a closed connected surface with genus $g$, then it admits a triangulation whose exterior algebraic shifting is a homology lex-segment complex.
\end{corollary}
\begin{proof}
  By~\cite[Theorem 1.2]{keehn24-exterior} there is a triangulation of the torus on $10$ vertices and a triangulation of the projective plane on $7$ vertices whose exterior algebraic shifting are homology lex-segments, see Figure \ref{fig:triangulations}(2-3).
  By Corollary~\ref{c:gluehomlex} attaching iteratively $g$ copies of this triangulation along $2$-faces we obtain a triangulation of $S_g$, or $N_g$, that is a homology lex-segment complex.
\end{proof}

\noindent
{\bf Edge contraction and vertex-split.}
Let $K$ be a simplicial complex and $e=\{u,v\}\in K$. We say that $K'$ is \emph{obtained from $K$ by contracting the edge $e$} if $K'$ is obtained from $K$ by replacing every occurrence of the vertex $u$ with the vertex $v$ (and removing duplicated faces).
The inverse operation of an edge contraction is called a \emph{vertex-split}.
In the case that $K$ is a surface triangulation, an edge contractions results in a triangulation of the same surface if an only if the contracted edge is not part of a missing triangle, namely a triangle not in $K$ whose boundary is contained in $K$. The triangulation of the torus in Figure~\ref{fig:triangulations}(1) is obtained from Figure~\ref{fig:triangulations}(2) by means of three edge contractions. Namely, $\{2,8\},\{4,10\}$ and $\{1,9\}$.
Performing a vertex-split on a simplicial complex can only make its exterior algebraic shifting simpler.
To state this precisely we set $\tail(K,A) = \{B \in K \colon A \leq B\}$.

\begin{proposition}[\cite{keehn24-exterior}]
  \label{p:splittail}
  Let $K$ be a surface triangulation and $K'$ obtained from $K$ by a vertex-split.
  Then, for $m\geq 3$ we have that
  $$|\tail(\exshift{K},\{m+1,m+2\})| \geq |\tail(\exshift{K'}, \{m+1,m+2\})|,~\text{and}$$
  $$|\tail(\exshift{K},\{1,m+1,m+2\})| \geq |\tail(\exshift{K'}, \{1,m+1,m+2\})|.$$
\end{proposition}

\begin{proposition}[\cite{bulavka23-rigidity}]
  \label{p:splitvolume}
  Let $K$ be a surface triangulation on $n$ vertices and $K'$ obtained from $K$ by a vertex-split.
  If $\{1,3,n\}\in \exshift K$, then $\{1,3,n+1\}\in \exshift{K'}$.
\end{proposition}

\begin{corollary}
\label{c:splithlex}
  Let $K$ be a surface triangulation and $K'$ obtained from $K$ by a vertex-split.
  If $\exshift K$ is a homology lex-segment complex, then $\exshift{K'}$ is as well.
\end{corollary}
\begin{proof}
  On the one hand, Proposition~\ref{p:splittail} implies that $|\tail(\exshift{K'},\{5,6\})|\leq |\tail(\exshift{K},\{5,6\})|=0.$
  On the other hand, Proposition~\ref{p:splitvolume} says that $\{1,3,n'\}\in \exshift{K'}$, where $n'$ is the number of vertices in $K'$.
  Then, the tails $\tail(\exshift{K'},\{4,5\})$ and $\tail(\exshift{K'},\{1,4,5\})$ are totally ordered sets 
  of same sizes as the corresponding tails w.r.t. $K$, and hence equal to the corresponding tails w.r.t. $K$
  (these sizes are determined solely by the genus of the surface). To conclude, $\exshift{K'}$ is a homology lex-segment complex.
\end{proof}

The following lemma is similar to \cite[Lemma 4.10]{keehn24-exterior}, which in turn is similar to Whiteley's vertex splitting lemma~\cite{whiteley90-split}.  
\begin{lemma}
  \label{l:interior-edge}
  Let $K$ be a triangulation of a disc with at least one edge $e$ adjacent to at least one interior vertex. 
  Then, there exists a contractible edge incident to at least one interior vertex.
  Moreover, if both endpoints of $e$ are interior then there exists a contractible edge incident to two interior vertices.
\end{lemma}
\begin{proof}
  Let $K_0 = K$ and $e_0=\{u,v\}$ an edge incident to an interior vertex.
  If $e_i$ is not part of a missing triangle then it is the desired edge.
  Otherwise, $e_i$ is part of the boundary of a missing triangle $T_i$.
  Let $K_{i+1}$ denote the disc subcomplex of $K_i$ with boundary $\partial T_i$.
  Since each endpoint of $e_i$ is adjacent with an interior vertex of $K_{i+1}$, 
  select one of these edge and denote it $e_{i+1}$. Because the triangulation is finite
  we eventually encounter an edge that is not part of any missing triangle, which is a contractible edge as desired.
\end{proof}

\subsection{The Delaunay complex}
\noindent
{\bf Riemannian surfaces.}
An \emph{embedding} of a simplicial complex $K$ into a surface $S$ is a continuous injective map from its geometric realization into $S$. We will denote by $|K|$ the image of the embedding of $K$ in $S$. If this map is a homeomorphism then $K$ is said to be a \emph{triangulation} of $S$. If in addition $S$ has a smooth structure then an embedding is called (piecewise) smooth if restricting it to each edge of $K$ gives rise to a (piecewise) smooth map.
The following is well-known.
\begin{theorem}{\cite{hatcher13-triangsurf}}
\label{t:smooth-emb}
Let $K$ be a simplicial complex and $S$ be a surface with a smooth structure. If $K$ is a triangulation of $S$ then it can be smoothly embedded into $S$.
\end{theorem}

Now, let $S$ be a connected closed Riemannian surface, i.e., compact and without boundary. By the Hopf-Rinow theorem~\cite{docarmo92-riemannian} every two points on $S$ are connected by a minimal length geodesic. For $p,q\in S$ let us denote by $d(p,q)$ the distance between $p$ and $q$ given by a minimal length geodesic joining them.
For $X\subseteq S$ and $\epsilon >0$ we will denote $B(X,\epsilon) = \{q\in S \colon d(X,q)<\epsilon \}$ the open neighborhood of a set $X$ of radius $\epsilon$.
The \emph{injectivity radius at $p\in S$}, denoted by $\rinj(p)$,
is the largest real number $r$ such that whenever $d(p,q)<r$, then there exists
a unique minimal geodesic from $p$ to $q$. The \emph{injectivity radius of $S$} is defined as $\rinj(S) = \inf_{p\in S} \rinj(p)$.
A subset $A\subseteq S$ is said to be \emph{strongly convex} if for every two elements $p,q\in A$ there exists a unique minimal geodesic joining them, this geodesic is contained in $A$ and there is no other geodesic in $A$ joining $p$ and $q$. 
The \emph{strong convexity radius at $p$} is defined as $$\rscv(p) = \sup \{r \colon B(p,r')~\text{is strongly convex for every}~r' < r \}.$$
The strong convexity radius of $S$ is $\rscv(S) = \inf_{p\in S}\rscv(x)$ and is known to be positive~\cite[1.9.9]{klingenberg95-riemannian}.

\vspace{3mm}
\noindent
{\bf Tubular neighborhood.} For each point $p\in S$ the exponential map at $p$ is the smooth map $\exp_p\colon T_pS \rightarrow S$
that assigns for every vector $X\in T_pS$ the point $\gamma(\norm{X})$ where $\gamma$ is the geodesic starting
at $p$ with tangent vector $\gamma'(0)=X/\norm{X}$. Here the norm is given by the Riemannian metric on $T_pS$.
Now, let $\gamma$ be a smooth curve in $S$ and $\epsilon>0$.
At each point $p=\gamma(t)$ let $X\in T_pS$ be an orthonormal vector with respect to $\gamma'(t)$ (there are two choices, one the minus of the other, but both result in the same tubular  neighborhood, defined next) and set
$$I(p,\epsilon) = \{ \exp_p(tX) \colon t\in (-\epsilon,\epsilon)\} \text{ and } I(\gamma,\epsilon) = \bigcup_{p\in \gamma} I(p,\epsilon).$$
There exists $\epsilon>0$ such that the sets $I(p,\epsilon)$ are all disjoint for $p\in \gamma$~\cite[Proposition 7.26]{oneill83-tubular}.
In this case $I(\gamma,\epsilon)$ is called the \emph{tubular neighborhood of $\gamma$ with radius $\epsilon$}.
For a simple curve $\gamma$ with distinct endpoints $a$ and $b$ we extend its tubular neighborhood by adding a cap on
each end point. Concretely, we set $\tilde{I}(\gamma,\epsilon') = B(a,\epsilon')\cup B(b,\epsilon')\cup I(\gamma, \epsilon')$
where $0<\epsilon'\leq \epsilon$ is chosen to be small enough in order for the caps to be disjoint embedded discs.

\vspace{3mm}
\noindent
{\bf Delaunay complex.} Let $P\subseteq S$ be a finite subset,
\emph{the Delaunay complex of $P$ in $S$}, denoted by $\Del(P)$,
is the simplicial complex whose faces are given by subsets $A \subseteq P$
for which there exists an open ball $B(x,\epsilon)\subseteq S$ disjoint from $P$
such that $A \subseteq cl (B(x,\epsilon))$~\cite{delaunay34-sphere}.
Although $\Del(P)$ always defines a simplicial complex, its geometric realization is in general not homeomorphic to $S$.
In the case of surfaces a density condition suffices to ensure that for a subset $P\subseteq S$ locally in general position, $\Del(P)$ is a smoothly embedded triangulation of $S$. 
We make this requirement precise now. For it set $\rdel(S) = \min\{\rinj(S)/6, \rscv(S)\}$.
A subset $P\subseteq S$ is \emph{locally in general position} if no point is on the minimal geodesic between other two points at distance less than $\rdel(S)$, and no four points are simultaneously on the boundary of a ball of radius less than $\rdel(S)$.
A set of points $P\subseteq S$ is \emph{$\rho$-dense} if there is at least one point of $P$
in any ball in $S$ of radius greater or equal to $\rho$.
The existence of a Delaunay triangulation is given by the following theorem,
see also~\cite[Theorem 2]{dyer08-surface}.

\begin{theorem}[\cite{leibon99-delaunay}]
  \label{t:deltriang}
  Let $S$ be a closed connected Riemannian surface and $\rho \leq \rdel(S)$.
  If $P\subseteq S$ is a $\rho$-dense finite subset locally in general position,
  then $\Del(P)$ is a geodesically embedded triangulation of $S$. 
\end{theorem}

In order to control the behaviour of Delaunay edges we will repeatedly use the following bound on their length.

\begin{lemma}
  Let $P\subseteq S$ be a $\rho$-sense finite subset such that $\Del(P)$ is a geodesically embedded triangulation of $S$.
  Then, for each edge $e\in \Del(P)$ its embedding $\gamma_e$ has length strictly less than $2\rho$.
\end{lemma}
\begin{proof}
  Let $B(x,r)$ be a ball witnessing that $e=\{u,v\}$ is a edge.
  Then, $r< \rho$ since otherwise it contains $B(x,\rho)$ whose intersection with $P$ is empty contradicting the $\rho$-dense assumption.
  Then, $\length(\gamma_e)=d(u,v)\leq d(u,x)+d(x,v) = 2r < 2\rho$ as desired.
\end{proof}

\section{The refinement of a triangulation}
~\label{sec:Universality}
In this section we prove our universality Theorem~\ref{t:delaunayref}.
The proof of this theorem proceeds in two steps.
The first step, Proposition~\ref{p:refinement}, provides a density $\rho>0$ as well as a \emph{partition} of the corresponding Delaunay triangulation,  
showing that combinatorially a $\rho$-dense Delaunay triangulation refines a regular cell complex that refines the original triangulation $K$. 
The second step, Proposition~\ref{p:minor}, verifies that the edges of a subdivision of $K$ can be contracted in order to reach $K$.
The following is the definition of a partition we require, see~\cite[Sect. 15]{munkres84-elements}.

\begin{definition}
    Let $S$ be a surface and $K,K'$ be two simplicial complexes embedded in $S$. The embedding \emph{$|K'|$ is a subdivision of $|K|$} if the embedding of every face of $K'$ is contained in the embedding of some face of $K$, and the embedding of every face of $K$ is given by a union of the embeddings of finitely many faces from $K'$.
    For a face $A\in K$ we denote by $K'_A$ the subcomplex of $K'$ whose union is $|K'_A| = |A|\subseteq S$.
\end{definition}

Recall that for an embedded simplicial complex $K$, the \emph{open star of a face $A\in K$} is the open subset $\st^\circ(A,K)\subseteq |K|$ given by the union of the relative interiors of the faces containing $A$, including $A$ itself. The \emph{closed star $\st(A,K)$} is the closed set $\cl \st^\circ(A,K)$ or equivalently it is the union of the relative closures of all the faces containing $A$~\cite{munkres84-elements}.

\begin{definition}
    Let $K$ be a simplicial complex embedded in $S$. A family of closed discs with smooth boundaries $\mathcal{B}_0 = \{B_v \subseteq B'_v \subseteq B''_v \colon v\in K_0 \}$ is a \emph{layered vertex cover} if 
    \begin{enumerate}
        \item $B_v\subseteq \interior B'_v$ and $B'_v\subseteq \interior B''_v$.
        \item $B''_v\cap B''_u = \emptyset$ for every pair of distinct vertices $v,u\in K_0$.
        \item $v \in B_v\subseteq B'_v\subseteq B_v'' \subseteq \st^\circ(v,K)$.
    \end{enumerate}
\end{definition}
Later we will make the middle layer wide enough to guarantee the conclusion in Lemma~\ref{l:vertexedge-containment}, which is required in the proof of ``the first step" Proposition~\ref{p:refinement}.
%%%%%%%%%%%%%%%%%%%%%%%%%%%%

The following intuitive and simple lemma guarantees that a layered vertex cover exists. 
\begin{lemma}
\label{l:vertexdisc}
Let $S$ be a Riemannian surface, $K$ an embedded triangulation of $S$. Then, there exists $\epsilon>0$ such that $\{B(v,\epsilon)\subseteq B(v,2\epsilon) \subseteq B(v,3\epsilon) \colon v\in K_0 \}$ is a layered vertex cover.
\end{lemma}
\begin{proof}    
    First, for any pair of distinct vertices $u,v\in K_0$ we have that $d(u,v)>0$. Since $K$ has finitely many vertices by setting $\epsilon_2 = \min_{u\neq v\in K_0}d(u,v)/2$ the second condition is satisfies.

    Since the open star $\st^\circ(v,K)$ is an open subset in $|K|$ and the latter is homeomorphic to $S$ we can conclude that $\st^\circ(v,K)$ is an open subset in $S$. Moreover, since it contains $v$ there exists $\epsilon_{3,v}>0$ such that $B(v,\epsilon_{3,v})\subseteq \st^\circ(v,K)$.
    Let $\epsilon_3=\min_{v\in K_0} \epsilon_{3,v}$.

    Finally, setting $\epsilon=\min\{\rinj/2,\epsilon_2,\epsilon_3\}/3$ we guarantee that all the desired conditions are satisfied.
\end{proof}

\begin{definition}
    Let $K$ be a simplicial complex piecewise smoothly embedded in $S$ and $\mathcal{B}_0$ a layered vertex cover. The embedding is \emph{transversal with respect to $\mathcal{B}_0$} if
    \begin{enumerate}
        \item for each edge $e=\{u,v\}$ its embedding $\gamma_e$ intersects $\partial B_v$ exactly once, at a point denoted $a_{e,v}$, and, 
        \item for each edge $e=\{u,v\}$ its embedding $\gamma_e$ is given by a concatenation $\gamma_{e,u}\circ \widetilde{\gamma_e}\circ \gamma_{e,v}$ where $\gamma_{e,u}$ is a smooth curve in $B_u$ connecting $a_{e,u}$ and $u$, $\gamma_{e,v}$ is smooth curve in $B_v$ connecting $v$ to $a_{e,v}$ and $\widetilde{\gamma_e}$ is a smooth curve connecting $a_{e,v}$ and $a_{e,u}$.
    \end{enumerate}
\end{definition}

In the following two lemmas we show the existence of a transversal embedding with respect to the layered vertex cover. 
First, we show that the cyclic order of intersections in a neighborhood of a vertex coincides with the one in its link. For it we recall that the \emph{link of a vertex $v$ in $K$} is the simplicial complex $\lk(v,K)=\{A\in K\colon v\notin A, A\cup \{v\}\in K\}$.

\begin{lemma}
\label{l:incidenceorder}
Let $S$ be a closed surface, $K$ an embedded triangulation of $S$,
$B \subseteq \st^\circ(v,K)$ a disc with smooth boundary containing $v$ in its interior.
For each edge $e=\{u,v\}$ let $a_u$ denote the first intersection of $\partial B$ along $\gamma_e$ when starting from $u$.
Then, the cyclic order of $\{a_u\colon \{v,u\}\in K\}$ coincides with that of the vertices in the link $\lk(v,K)$.
\end{lemma}
\begin{proof}
Consider a counterclockwise orientation on the vertex  $v$ and the boundary $\partial B$.
Let $e=\{u,v\},f=\{u',v\},g=\{u'',v\}$ be consecutive edges incident on $v$ in counterclockwise order with $a_u,a_{u'},a_{u''}$
being their corresponding first intersection with $\partial B$ as in the statement.
We need to verify that $a_{u'}$ is between $a_u$ and $a_{u''}$ in the counterclockwise order along
$\partial B$. Since $K$ is a triangulation of a surface, the closed star $\st(v,K)$ is
a disc that is split into two components by the concatenation of $e$ and $g$.
Consequently, $f$ intersects $\partial B$ within the same connected component as its endpoint is, which is between $a_u$ and $a_{u''}$ in the desired order.
\end{proof}

\begin{figure}[!htb]
\minipage{0.30\textwidth}
\includegraphics[width=\linewidth]{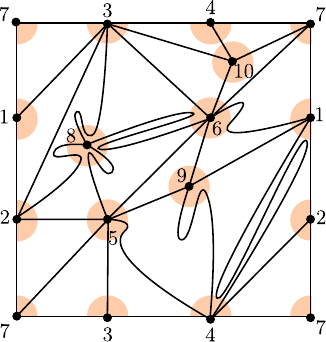}
\centering (1)
\endminipage\hfill
\minipage{0.30\textwidth}
\includegraphics[width=\linewidth]{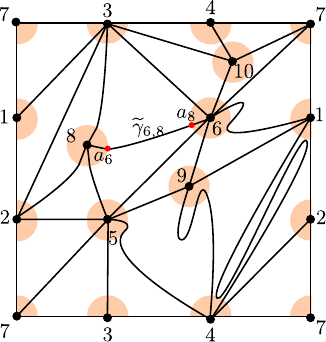}
\centering (2)
\endminipage\hfill
\minipage{0.30\textwidth}%
\includegraphics[width=\linewidth]{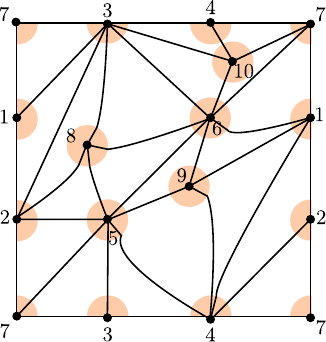}
\centering (3)
\endminipage
\caption{(1) Embedding of the flat torus with intersection in a neighborhood of a vertex. (2) Embedding in the flat torus with straightening of edges adjacent to vertex $8$. (3) A transversal embedding in the flat torus with respect to a layered vertex cover.}
\label{fig:goodemb}
\end{figure}

\begin{lemma}
    \label{l:good-emb}
    Let $K$ be a triangulation of $S$, then $K$ admits a piecewise smooth embedding in $S$ that is transversal with respect to some layered vertex cover.
\end{lemma}
\begin{proof}
    By Theorem~\ref{t:smooth-emb} $K$ admits a smooth embedding and 
    let $\mathcal{B}_0$ be the layered vertex cover given by Lemma~\ref{l:vertexdisc}. 
    For each edge $e=\{u,v\} \in K$ let $\gamma_e$ denote its embedding in $S$.
Consider $a_u$ the first intersection point of $\partial B_v$ with $\gamma_e$ when traversed starting from the vertex $u$.
By Lemma~\ref{l:incidenceorder} the cyclic order of collection $\{a_u\colon \{u,v\} \in K\}$ coincides with that of the vertices in $\lk(v,K)$.
As $S$ is Riemannian, by taking $\epsilon>0$ sufficiently small, in the disc $B_v=B(v,\epsilon)$ the unique geodesics from any two of its boundary points to $v$ intersect only in $v$; now we replace the embedding $\gamma_e$ after $a_u$ with the smooth geodesic curve joining $a_u$ with $v$, see Figures~\ref{fig:goodemb}(1) and \ref{fig:goodemb}(2).
Since the cyclic order around $v$ is preserved the resulting triangulation, with the tips of the edges straightened, coincides with $K$.
We repeat this procedure at every vertex of $K$, see Figure~\ref{fig:goodemb}(3).
Finally, let $\widetilde{\gamma}_e$ denote the segment of $\gamma_e$ joining its unique intersection points in $B_v$ and $B_u$.
\end{proof}

\begin{definition}
    Let $K$ be a piecewise smooth embedding into $S$ that is transversal with respect to the layered vertex cover $\mathcal{B}_0$. A family of closed discs $\mathcal{B}_1 = \{B_e \subseteq S\colon e\in K_1\}$ is an \emph{edge cover} if it the satisfies the following four conditions: 
    \begin{enumerate}
        \item $\widetilde{\gamma_e}\subseteq \interior B_e$.
        \item $B_e \subseteq \st^\circ(e,K)$.
        \item $B_e\cap B_{e'} = \emptyset$ for every pair of distinct edges $e,e'\in K_1$.
        \item $B_e\cap B''_v = \emptyset$ if $v\notin e$.
    \end{enumerate}
\end{definition}

\begin{lemma}
\label{l:edge-neigh}
    Let $K$ be a simplicial complex piecewise smoothly embedded in $S$ that is transversal with respect to the layered vertex cover $\mathcal{B}_0$. Then, there exists $\delta>0$ such that $\{B(\widetilde{\gamma_e},\delta) \colon e\in K_1\}$ is an edge cover.
\end{lemma}
\begin{proof}
    The curve $\widetilde{\gamma_e}$ is contained in the interior of the tubular neighborhood with caps $\interior\widetilde{I}(\widetilde{\gamma_e})$ and consequently by setting $\delta \leq d(\widetilde{\gamma_e},\partial \widetilde{I}(\gamma_e)) > 0$ we guarantee that the resulting set is an embedded disc. Property (1) holds by construction, and small enough $\delta$ clearly guarantees (2) as well. 
    
    Since $\widetilde{\gamma_e}$ and $\widetilde{\gamma_{e'}}$ are disjoint for every pair of distinct edges $e,e'\in K$, then by setting $\delta 
    \leq d(\widetilde{\gamma_e},\widetilde{\gamma_{e'}})/2$ property (3) holds. 
    
    Finally, as $\gamma_e \cap B''_v = \emptyset$ for $v\notin e$ since $B''_v\subseteq \st^\circ(v,K)$ then by setting $\delta \leq d(\widetilde{\gamma_e},B''_v)>0$ also property (4) is satisfied.
\end{proof}

\begin{lemma}
    \label{l:del-neigh}
    Let $S$ be a closed connected Riemannian surface, 
    let $X \subseteq D\subseteq S$ be a pair of discs in $S$ and $P\subseteq S$ a $\rho$-dense finite set such that $\Del(P)$ is a triangulation of $S$. 
    If $2\rho < \min \{d(X,\partial D), \rscv(S) \}$, then the restriction of $\Del(P)$ to the triangles all whose vertices are in $D$, denoted by $\Del_D$, 
    contains a unique maximal disc subcomplex $\Del_{D,X}$ containing $X$. Further, all the boundary vertices of $\Del_{D,X}$ lie outside of $X$.
\end{lemma}

\begin{proof}
As $2\rho < \rscv(S)$,
every cycle on the edges of $\Del_D$ is fully contained in the disc $D$,
hence so is its inner connected component disc $C$. Since $\Del(P)$ is a triangulation of $S$, it has a subcomplex that triangulates $C$, and its triangles are present in $\Del_D$ because they are contained in $D$.
Thus, $\Del_D$ is a simply connected planar pure $2$-dimensional complex, hence a \emph{cacti}, see e.g.~\cite{cacti-cohen2002homotopy}, namely the union of maximal simplicial discs, every two of them are either disjoint or intersect in a single vertex, and the bipartite graph whose vertices correspond to the set $A$ of discs and the set $B$ of their intersection points and whose edges $ab$ correspond to incidences $b\in a$ for $(a,b)\in A\times B$, is a forest.

Denote by $\Del_{D,X}$ the maximal disc in the  decomposition that contains $X$. 
It is left to show that $\Del_{D,X}$ is well defined. Indeed, for every point 
$x\in X$, each vertex in a triangle $A$ in $\Del(P)$ containing $x$ is of distance at most $2\rho$ from $x$. Since $2\rho <\rscv$ the embedding of $A$ is contained in $B(x,2\rho)$ since this last one is strongly convex. Moreover, given that $2\rho < d(x,\partial D)$ the ball $B(x,2\rho)$, and consequently $A$ as well, is contained in $D$. Thus, the triangle $A$ is in $\Del_D$, and consequently there exists a unique disc component $\Del_{D,X}$ in the  that contains $X$. 

Further, as $2\rho < d(X,\partial D)$, for a vertex $v\in\Del(P)$ lying in $X$, all its neighbors are in $D$, hence such $v$ can not be on the boundary of $\Del_{D,X}$.
\end{proof}

\begin{lemma}
    \label{l:vertexedge-containment}
    Let $B\subseteq B'\subseteq B''\subseteq S$ closed discs with smooth boundary such that $B\subseteq \interior B'$, $B'\subseteq \interior B''$, $\rho>0$ and $P\subseteq S$ a $\rho$-dense finite point set such that $\Del(P)$ is a triangulation of $S$.
    If $2\rho < \min \{d(B,\partial B'), d(B',\partial B''), \rscv\}$, then 
    the vertex set of $\partial \Del_{B'',B}$ is contained in $B''\setminus B'$ and every edge between boundary vertices of $\Del_{B'',B}$ is 
    disjoint from $B$.
\end{lemma}
\begin{proof}

As $B\subseteq B'$, 
the maximal discs $\Del_{B'',B}$ and $\Del_{B'',B'}$ from Lemma~ \ref{l:del-neigh} satisfy, by definition, that they are equal, and their boundary vertices lie in $B''\setminus B'$. 

    Now, let $e=\{u,v\}$ be a Delaunay edge with endpoints in $B''\setminus B'$. Then,  
    $\gamma_e \cap B = \emptyset$  
    since $d(B,\partial B')>2\rho$ and all edges have length at most $2\rho$.
\end{proof}

\begin{figure}[!htb]
\minipage{0.30\textwidth}
\includegraphics[width=\linewidth]{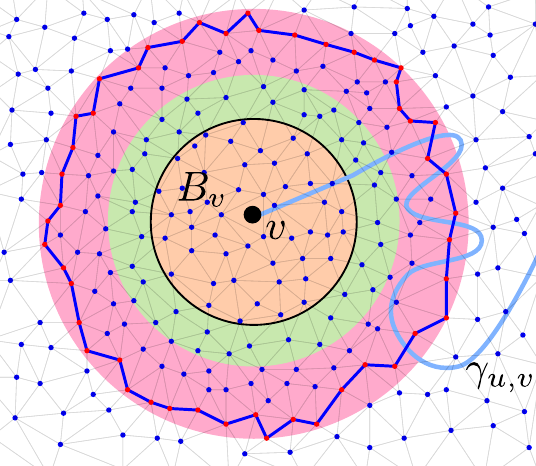}
\centering (1)
\endminipage\hfill
\minipage{0.30\textwidth}
\includegraphics[width=\linewidth]{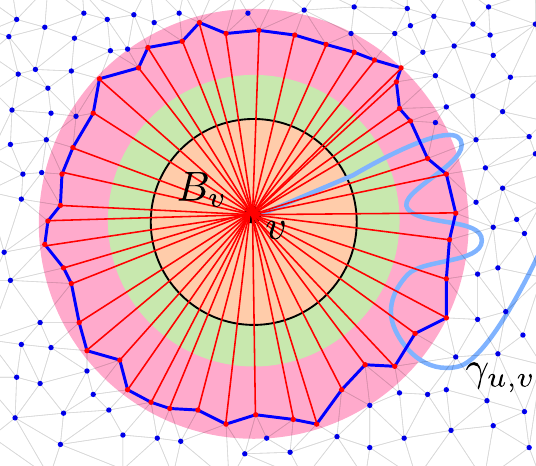}
\centering (2)
\endminipage\hfill
\minipage{0.30\textwidth}%
\includegraphics[width=\linewidth]{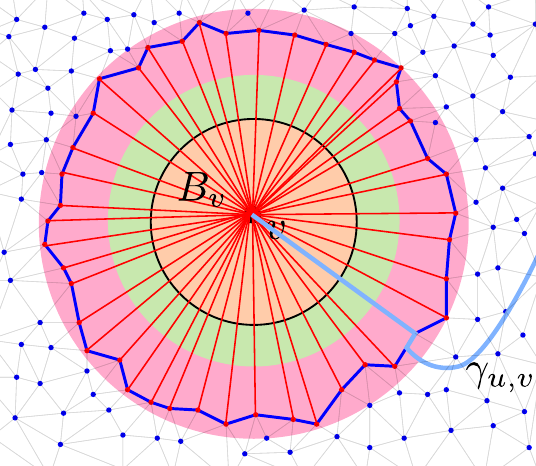}
\centering (3)
\endminipage
\caption{
(1) The Delaunay triangulation in a neighborhood of the vertex $v$, where layered vertex cover of $v$ is depicted. The vertices of $\partial \Del_{B''_v,B_v}$ (marked in red) and its edges (marked in blue) are disjoint from $B_v$. The embedding $\gamma_{u,v}$ is shown in light blue.
(2) The triangulation obtained after contracting edges inside $D_v$. The resulting simplicial complex is a cone. Note that the vertices on the boundary of $D_v$ form a strict subset of those on the boundary of $\Del_{B''_v,B_v}$.
(3) The embedding of the edge $\gamma_{u,v}$ is modified: after intersecting $\partial D_v$, it continues along the shortest path to the vertex $v$.}
\label{fig:delneigh}
\end{figure}

\begin{figure}
    \centering
    \includegraphics[width=1\linewidth]{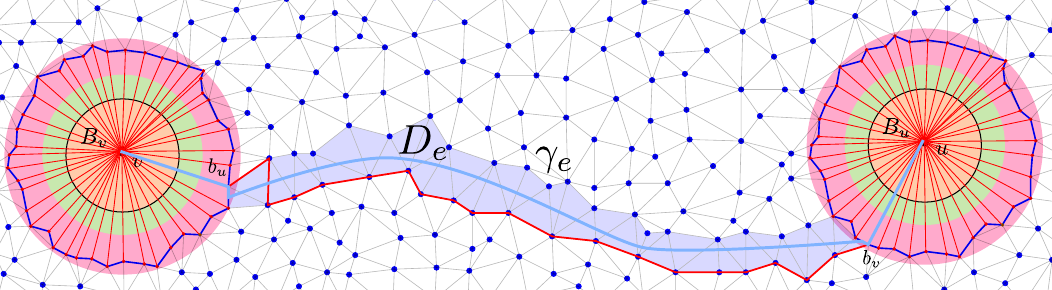}
    \caption{Edge approximation in a Delaunay triangulation inside the edge neighborhood $\Del_e$.}
    \label{fig:deledge}
\end{figure}

\begin{proposition}
\label{p:refinement}
    Let $S$ be a closed Riemannian manifold and $K$ a triangulation of $S$. There exist a piecewise smooth embedding of $K$ in $S$ and a density $\rho>0$ such that the following holds. If $P\subseteq S$ is a $\rho$-dense finite set locally in general position, then $\Del(P)$ admits a sequence of edge contractions reaching a subdivision of $K$.
\end{proposition}
\begin{proof}
    First, by Lemma~\ref{l:good-emb} the triangulation $K$ admits a piecewise smooth embedding in $S$ that is transversal with respect to some layered vertex cover $\mathcal{B}_0$. In addition, let $\mathcal{B}_1$ be the edge cover guaranteed by Lemma~\ref{l:edge-neigh}.
    Finally, set $\rho >0$ small enough, to be specified below, only depending on $S$, the embedding of $K$, the layered vertex cover $\mathcal{B}_0$ and the edge cover $\mathcal{B}_1$.

    Let $\rho < \rdel(S)$ and $P\subseteq S$ a $\rho$-dense finite subset locally in general position. Then, Theorem~\ref{t:deltriang} implies that the simplicial complex $\Del(P)$ is a triangulation of $S$.

    Call an edge in a triangulated disc \emph{diagonal} if both its vertices belong to the boundary of the disc.
    Now, let $\Del_{B''_v,B_v}$ be the disc guaranteed by Lemma~\ref{l:del-neigh}, see Figure~\ref{fig:delneigh}(1). We proceed to select a subcomplex of $\Del_{B''_v,B_v}$ without diagonal edges while still containing $B_v$ in its interior. For it, let $e$ be a diagonal edge in $\Del_{B''_v,B_v}$. Then $e$ splits this complex into two connected components both of which are 
discs. 
    By taking $2\rho < \min \{d(B_v,\partial B_v'), d(B_v',\partial B_v''), \rscv(S)\}$, Lemma~\ref{l:vertexedge-containment} guarantees that the embedding $\gamma_e$ does not intersect $B_v$. We iterate this procedure with the component having $B_v$ in its interior. The resulting triangulated disc $D$ has no diagonal edges and contains $B_v$ in its interior.
    In $D$ we proceed to iteratively contract edges whose both endpoints are in the interior of $D$, see Lemma~\ref{l:interior-edge}. The resulting complex $D_v$ is a star with boundary $\partial D$ and we identify the apex vertex with $v$, see Figure~\ref{fig:delneigh}(2) Let $K'$ denote the result of applying this sequence of edge contractions to $\Del(P)$ for every vertex $v\in K_0$.

    Next we modify the embedding of $K$ 
    so that its edges are realized on the $1$-skeleton of $K'$.
    For it, let $e=\{u,v\}\in K_1$ and $\Del_{B_e,\widetilde{\gamma_e}} \subseteq \Del(P)$ given by Lemma~\ref{l:del-neigh}\footnote{The definition of $\Del_{B,X}$ applies to every connected subspace $X$ contained in the interior of a disc $B$. Indeed, as $X$ is far enough from the boundary of $B$ then an open neighborhood of $X$ is contained in the  $\Del(P)\cap B$, hence $X$ is contained in a maximal disc component of the  -- this is $\Del_{B,X}$.}, where we have applied the lemma to a small enough neighborhood of $\widetilde{\gamma_e}$. 
    Since $\widetilde{\gamma_e}$ intersects $\partial D_u$ and $\partial D_v$ so does $\Del_{B_e,\widetilde{\gamma_e}}$. Let $\gamma'_e$  be a shortest path in the graph metric contained in the $1$-skeleton of $\Del_{B_e,\widetilde{\gamma_e}}$ between $\partial D_u$ and $\partial D_v$ with endpoints $b_{e,u}$ and $b_{e,v}$ respectively, see Figure~\ref{fig:deledge}. 
We need to make sure that for $e=vu$ and $e''=vu''$ the vertices $b_{e,v}$ and $b_{e'',v}$ in $D_v$ are distinct. Indeed, by taking $4\rho< d(\gamma_e\cap(B''_v\setminus B_v), \gamma_{e''}\cap(B''_v\setminus B_v))$ for every two edges $e,e''$ both containing $v$ as a vertex, we guarantee that $b_{e,v}$ and $b_{e'',v}$ are distinct.
    Finally, replace the embedding $\gamma_e$ by the concatenation $\gamma'_{e,u}\circ \gamma'_e\circ \gamma'_{e,v}$ where $\gamma'_{e,u}$ is the unique 
    edge in $D_u$ from $b_{e,u}$ to $u$ and similarly $\gamma'_{e,v}$ is the unique 
    edge in $D_v$ from $b_{e,v}$ to $v$, see Figures~\ref{fig:delneigh}(3) and~\ref{fig:deledge} . Since the cyclic order at every vertex of $K$ is preserved, the embedded simplicial complex coincides with $K$ (combinatorially) and it is subdivided by $|K'|$.
\end{proof}

Recall that if $K'$ is a subdivision of $K$, given a face $A\in K$ we denote by $K'_A$ the subcomplex of $K'$ subdividing $A$.

\begin{proposition}
  \label{p:minor}
  Let $K,K'$ be two embedded triangulations of a surface $S$ such that $|K'|$ is a subdivision of $|K|$.
  Then, there exists a sequence of edge contractions from $K'$ to $K$.
\end{proposition}
\begin{proof}
   We proceed by induction on the number of edges of $K'$.
   We split the analysis into two cases, in each case we perform an edge contraction to reduce the number of edges and proceed by induction.
   If $K_0=K'_0$ then the simplicial complexes coincide. Let us assume then that $K'$ has more vertices and consequently there exists a face $A\in K$ such that $K'_A$ has at least one interior vertex. We split the analysis into two cases, either there exists such a triangle $A$, or else $A$ must be an edge.
   
   \noindent
   {\bf Case 1:} If $A$ is a triangle, we let $K'_A$ be the subdivision of a $2$-face with an interior vertex. Then, by Lemma~\ref{l:interior-edge} there exists a contractible edge incident to an interior vertex.
   
   \noindent
   {\bf Case 2:}  Suppose
    that there are no more $2$-faces of $K$ whose subdivision has an interior vertex. Let $K'_e$ be the subdivision of an edge $e\in K$ with a vertex in its relative interior and let $f=\{a,b\}\in K'_e$ an edge. If $f$ is contractible, then we contract it and proceed by induction. Otherwise, $f$ is part of a missing triangle $T=\{a,b,c\}$ in $K'$, hence $c$ must be in $K'_e$ as we excluded Case 1.

   The boundary of $T$ is contained in the disc $D$ formed by the embedded two triangles in $K$ that contain $e$. Thus, the boundary of $T$ bounds a disc $D_T$ realized inside $|D|$ which is a subcomplex of $K'$ with an interior vertex; all vertices of $D_T$ belong to $K'_e$.
   As $K'_e$ has finitely many vertices, proceeding in this manner for an edge $f'$ in $K'_e\cap D_T$ and so on, we finally find a contractible edge in $K'_e$,
similar to the proof of Lemma~\ref{l:interior-edge}.
\end{proof}

\section{Concentration of exterior algebraic shifting for Random Delaunay}~\label{sec:assDelaunay}
Let $\nu$ be a volume measure on $S$ and $U_S$ be the random variable on $S$ uniformly distributed with respect to $\nu$, i.e., for a $\nu$-measurable subset $V\subseteq S$ we have that $\P(U_S\in V)=\nu(V)/\nu(S)$. For $n\in \N$ let $P_n$ denote the random variable of picking $n$ (unlabelled) points from $S$ independently uniformly at random according to $\nu$. Equivalently, we can consider the stochastic process $(P_n)_{n\in \N}$ where at each step we pick a new point in $S$ uniformly at random and independent from the previous choices. We are interested in sampling points that are locally in general position (g.p. for short), i.e., (1) no point is on the minimal geodesic between other two points at distance less
than $\rdel(S)$, and (2) no four points are simultaneously on the boundary of a ball of radius less than $\rdel(S)$.
\begin{lemma}
  \label{l:genericdense}
  Let $S$ be a closed connected Riemannian surface, $\rho>0$ fixed, and $P_n$ as above. Then,  
   $P_n$ is almost-surely locally in g.p for every $n\in \mathbb N$, and a.a.s. $\rho$--dense.
\end{lemma}
\begin{proof}
  For the first part we proceed by induction on $n$. Denote by $E_n$ the event that $P_n$ is locally in g.p. Observe that $\P(E_n)=\P(E_n|E_{n-1})\P(E_{n-1})$ and that $\P(E_{n-1})=1$ by induction. Conditioned on $E_{n-1}$, the probability of $E_n$ is given by the probability of picking a point that is not in the union of the sets determined by the complements of the conditions $(1)$ and $(2)$ in the definition of g.p.above. Since these complements have $\nu$-measure $0$ in $S$ and there are finitely many of them, the conclusion follows.
  
  For the second part let $\{B(x,\rho/2)\colon x\in I\}$ be a finite open cover of $S$, which exists since $S$ is compact. Then, it is enough to guarantee that $P_n\cap B(x,\rho/2)\neq \emptyset$ for every $x\in I$. Indeed, if this is true, then for every $y\in S$ there exists $x\in I$ such that $y\in B(x,\rho/2)$ and $p\in P\cap B(x,\rho/2)$. Then, $d(y,p)\leq d(y,x)+d(x,p)<\rho$ as claimed. It follows that
  $$\P(P_n~\text{is $\rho$--dense}) \geq \P(\forall x\in I, P_n\cap B(x,\rho/2)\neq \emptyset) \geq 1 - \sum_{x\in I}\P(P_n\cap B(x,\rho/2)=\emptyset) 
                         \geq 1-|I|(1-\mu)^n,
  $$
  where 
  $ \mu = \min_{x\in I} \nu(B(x,\rho/2))/\nu(S)>0$. 
  The conclusion now follows by taking $n\to \infty$.
\end{proof}

\begin{proof}[Proof of Theorem~\ref{t:aasDelauney}.]
  Let $K$ be a triangulation of $S$ given by Corollary~\ref{c:hlextriang} and set $\rho > 0$ as required by Theorem~\ref{t:delaunayref} when applied to $K$ and $S$.
  By Lemma~\ref{l:genericdense} $P_n$ is a.a.s. $\rho$-dense and is locally in general position. Then, Theorem~\ref{t:delaunayref} guarantees that $\Del(P_n)$ admits a sequence of edge contractions reaching $K$. Since $\exshift{K}$ is a homology lex-segment then Corollary~\ref{c:splithlex} implies that $\exshift{\Del(P_n)}$ is a homology lex-segment as well.
\end{proof}

\section{Concentration of exterior algebraic shifting for the uniform model}~\label{sec:ass1dim}
\noindent
{\bf One dimensional complexes.}
Let $X$ be a compact one dimensional topological space.
It admits natural triangulations each given by some graph $G$. Let $\sd(G)$ denote the barycentric subdivision of $G$, namely the one obtained by introducing a new vertex at the interior of each edge of $G$, subdividing that edge into two edges.
The following proposition holds for exterior algebraic shifting over any field:
\begin{proposition}
    \label{p:graphsdlex}
    Let $G$ be a graph, then $\exshift{\sd G}$ is a homology lex-segment.
\end{proposition}
\begin{proof}
    First, we claim that $\tail(\exshift{\sd G},\{3,4\}) = \emptyset$. Otherwise there exists a connected subgraph $H$ of $\sd G$ that is a $2$-hypercycle and consequently is $3$-edge connected, see \cite[Theorem 5.4]{kalai85-hyperconnectivity}. However, this is not possible for if $e=\{u,v\}\in H$ and we assume without loss of generality that $u \in V(\sd G)\setminus V(G)$ then the degree of $u$ in $H$ is at most $2$.
    Since algebraic shifting outputs a shifted simplicial complex with identical Betti numbers we conclude that $\exshift{\sd G} = [n]\bigcup \{A\in \binom{[n-\beta_0]}{2}\colon A \leq \{2,\beta_1+2\}\}$ as wanted.
    \end{proof}

\begin{proof}[Proof of Theorem~\ref{t:aasUniform1-dim}]
    First, by Propositions \ref{p:graphsdlex} and \ref{p:splittail} with $m=2$, the event that $\exshift{U_n(|G|)}=\Delta(|G|,n)$ contains the event $E$  that every edge of $G$ is subdivided. 
    In addition, since the order in which we subdivide an edge is not important the model $U_n(|G|)$ can be viewed as a balls and bins model where the bins represent the edges of $G$ and the balls represent the new vertices. Therefore, the probability that $E$ does not occur is at most $|E|(1-1/|E|)^n$ which tends to $0$ as $n\to\infty.$
\end{proof}

\noindent
{\bf Two dimensional complexes.}
We show here that unlike the one-dimensional case, the uniform triangulation of the $2$-disc and of the disjoint union of two $2$-discs do not have a concentrated exterior algebraic shifting.
\begin{proposition}
    \label{p:discshift}
    Let $K$ be a triangulation of a disc with $n$ interior vertices and $m$ boundary vertices, then $$\exshift{K} = [n+m] \bigcup \left \{A\in \binom{[n+m]}{2} \colon A \leq \{3,3+n\} \right \} \cup \left \{A \in \binom{[n+m]}{3} \colon A \leq \{1,3,3+n\} \right \} .$$
\end{proposition}
\begin{proof}
    Since $K$ is part of a triangulation of a sphere, then $\{1,4,5\} \notin \exshift{K}$. Moreover, since the homology of $K$ is trivial, then there is no $2$-face whose smallest vertex is strictly larger than $1$. Since the $1$-skeleton of $K$ is $2$-hyperconnected, see~\cite{kalai85-hyperconnectivity}, then $\{2,n+m\} \in \exshift{K}$. Combining this with the fact that $K$ is Cohen-Macaulay, even shellable, its algebraic shifting is pure implying that $\{1,2,m+n\}\in \exshift{K}$ and $\{4,5\}\notin \exshift{K}$ . Thus, the remaining edges must be of the form $\{3,k\}$, and the remaining triangles of the form $\{1,3,k\}$,    
    forming initial lex-segments of edges and of triangles by the shiftedness of $\exshift{K}$. The claim now follows from the fact that algebraic shifting preserves the $f$-vector.
\end{proof}

\begin{corollary}
    \label{c:discsconcentration}
    Let $X$ be $2$-dimensional disc, then $\exshift{U_n(X)}$ is not concentrated. 
\end{corollary}
\begin{proof}
    By well-known estimates ~\cite{Tutte1962,goulden2004combinatorial} for the number of $n$-vertex triangulations of the disc with $m$ boundary vertices, it can be derived that for every fixed $m\ge 3$ there exists $\varepsilon_m>0$ such that
    \[
    \P(U_n(X) \text{ has $m$ boundary vertices}) =\varepsilon_m +o(1)\,\]
    as $n\to \infty$. 
    In consequence, the probability that the $2$-faces satisfy $\exshift{U_n(X)} = \left \{A \in \binom{[n+m]}{3} \colon A \leq \{1,3,3+n\} \right \}$, for every fixed $m$, is bounded away from $0$ as $n$ grows.
\end{proof}

\begin{lemma}
    \label{c:twodiscshift}
    Let $K_1,K_2$ be the triangulations of two disjoint discs with $n_1,n_2$ interior vertices and $m_1,m_2$ boundary vertices respectively. Set $n=n_1+n_2$ and $m=m_1+m_2$, then
    $$\exshift{K_1\sqcup K_2} = [n+m] \cup \{\{i,j\}\colon 1\leq i\leq 3,~i \leq t\leq n+m-i\}\cup \left \{A\in \binom{[n+m-2]}{3}\colon A\leq \{1,3,n-3\} \right \}.$$
\end{lemma}
\begin{proof}
    The proof follows by applying Proposition~\ref{p:discshift} and Theorem~\ref{t:simplex-union}.
\end{proof}

\begin{corollary}
    Let $X=X_1\sqcup X_2$ be two disjoint $2$-dimensional discs, then $\exshift{U_n(X)}$ is neither concentrated nor a homology lex-segment (see Appendix~\ref{apx:lex} for definition).
\end{corollary}
\begin{proof}
    Let $K_1,K_2$ and $n$ and $m$ as in Lemma~\ref{c:twodiscshift}, then we have that $\{2,n+m-1\}\notin \exshift{K_1\sqcup K_2}$ while $\{1,n+m-1\},\{3,n-3\}\in \exshift{K_1\sqcup K_2}$ showcasing that the complex is not an homology lex-segment. The lack of concentration follows from  Corollary~\ref{c:discsconcentration}.
\end{proof}

\noindent

\section{Concluding remarks}~\label{sec:Conclude}
Conjecture~\ref{conj:uniform_surface} states that the exterior shifting of the uniform triangulation $U_n(S_g)$ of a fixed orientable surface is a homology lex-segment. In particular, this implies that  $U_n(S_g)$ is a.a.s. $K_6$-free (as otherwise the shifting contains the edge $56$). 
To the best of our knowledge, this statement is not known, although~\cite{budzinski2022multi} shows that a fixed neighborhood around \textit{almost} every vertex is a.a.s. planar. We now suggest an even more far-reaching problem, which arises naturally as a possible approach to extending our proof of Theorem \ref{t:aasDelauney} to Conjecture~\ref{conj:uniform_surface}:

\begin{problem}
    Fix a genus $g$. Show that there exists a Riemannian metric for $S_g$ such that for every fixed $\rho>0$, $U_n(S_g)$ can a.a.s. be embedded in $S_g$  such that every edge has length less than $\rho$. 
\end{problem}

In addition, note that in order to prove the analogous Theorem~\ref{t:aasUniform1-dim} for graphs, we first showed the deterministic statement Proposition~\ref{p:graphsdlex} on their barycentric subdivision. We conjecture that its analog holds for barycentric subdivision of surfaces as well:   
\begin{conjecture}\label{conj:br(surface)}
    If $K$ is a surface triangulation, then $\exshift{\sd(K)}$ is a homology lex-segment.
\end{conjecture}
Conjecture~\ref{conj:br(surface)} follows from the following conjecture, for exterior algebraic shifting over any field:
\begin{conjecture}
[{\cite[Conj.5.1]{bulavka23-rigidity}} for the case of characteristic zero]
\label{conj:area-rig}
For every triangulation $K$ on $n$ vertices of a connected compact surface without boundary, 
$\{1,3,n\} \in \exshift{K}$. 
\end{conjecture}
\begin{proof} of Conjecture~\ref{conj:br(surface)} modulo 
Conjecture~\ref{conj:area-rig}: 
Let $K$ triangulate $S_g$ and denote $m=f_0(\sd(K))$.
The $1$-skeleton of $\sd(K)$ is $3$-hyperconnected, see e.g.~\cite[Thm.3.4.2]{nevo07-thesis}, and similarly over every field, hence $\{3,m\}\in \exshift{\sd(K)}_1$. To see that $\{5,6\}\notin \exshift{\sd(K)}_1$, it is enough to remove from $\sd(K)$ vertices of degree at most $4$ one by one until we reach a collection of isolated vertices; see Kalai~\cite[Lem.4.3(i)]{kalai85-hyperconnectivity} for shifting over a field of characteristic zero, and the same proof holds over every field. First we remove vertices at the barycenter of edges of $K$ (their degree is $4$), then we remove vertices at the barycenter of triangles of $K$ (their degree at this stage is $3$), to reach the original vertices of $K$, and no higher dimensional faces. As 
$\exshift{\sd(K)}$ is shifted and has the same number of edges as $\sd(K)$, we conclude that $\exshift{\sd(K)}_1=\{A\in \binom{[m]}{2}:\ A\le_{\lex} \{4,4+6g\} \}$.
Assuming Conjecture~\ref{conj:area-rig}, $\{1,3,m\}\in \exshift{\sd(K)}_2$, and as $\exshift{\sd(K)}$ is a shifted simplicial complex, not containing the edge $\{5,6\}$ with the same face and Betti numbers as $\sd(K)$ we conclude that $\exshift{\sd(K)}_2=\{A\in \binom{[m]}{3}:\ A\le_{\lex} \{1,4,4+4g\}\}\cup\{\{2,3,4\}\}$.
\end{proof}

We remark that in order to extend Theorem~\ref{t:aasDelauney} to fields of characteristics different from $0$ and $2$ we only need to show that Corollary~\ref{c:hlextriang} holds over such fields. We leave this last item as an open problem.
\begin{problem}
    Let $K$ be the triangulation of either the torus or the projective plane appearing in Figure~\ref{fig:triangulations}(2-3). Then for every prime $p$, the exterior algebraic shifting of $K$ over a field of characteristic $p$, $\Delta^{ex}_p(K)$, is a homology lex-segment.
\end{problem}

%%%%%%%%%%%%%%%%%%%%%%5555
\bibliographystyle{alpha}
\bibliography{main}

\appendix
\section{Homology lex-segment complexes}\label{apx:lex}
In analogy to the compressed complexes, constructed as collections of initial segments w.r.t. the reverse lexicographic order, and their modification by Bj\"{o}rner and Kalai~\cite{Bjorner-Kalai-ACTA} that realizes all (face,Betti)-vectors of simplicial complexes, here we construct their analog w.r.t. the lex order. The main difference is that w.r.t. lex order, the shadow of an initial segment need not be an initial segment. We remedy this by appropriate augmentation of the shadow, characterizing numerically which face vectors and (face,Betti)-vectors of simplicial complexes are realized  
by the defined \emph{lex-segment complexes} and homology \emph{lex-segment complexes}.

Let $N=[n]=\{1 < \cdots < n\}$ denote the first $n$ positive integers with their natural order.
For a family of sets $K \subseteq 2^N$ we will denote by $K_r$ its members of size $r+1$.
A family of sets $K \subseteq 2^N$ is a \emph{simplicial complex} if whenever $A\subseteq B\in K$ then $A\in K$ as well.

For two distinct sets $A,B\in \binom{N}{r}$ we say that \emph{$A$ is lexicographically smaller than $B$}, denoted by $A<B$, if $\min ((A\cup B)\setminus (A\cap B)) \in A$.
For two integers $m,r\geq 0$ we set $I(N,r,m)$ to denote the initial segment consisting of the first $m$ elements in $\binom{N}{r+1}$ with respect to the lexicographic order. 
Clearly, there exists a unique sequence of integers $0=a_{-1}<a_0< a_1< \cdots < a_r\leq n$ such that $$I(N,r,m) = \left \{A \in \binom{N}{r+1} \colon A \leq \{a_0,\dots,a_r \} \right \},$$

and $m$
is given in terms of the last element $\{a_0<\cdots<a_r\}$ by $$m=|I(N,r,m)| = \sum_{i=0}^r \sum_{j=a_{i-1}+1}^{a_i-1} \binom{n-j}{r-i} + 1.$$

The \emph{shadow} of a set family $\mathcal{F} \subseteq \binom{N}{r+1}$ is defined by $\partial \mathcal{F} = \{A\in \binom{N}{r} \colon A\subseteq B \in \mathcal{F}\}$.
The shadow of an initial segment is characterized in terms of the last element $\{a_0<\cdots <a_r\}$ as follows
$$\partial I(N,r,m) = 
\begin{cases} 
    \binom{N}{r} & \text{if}~a_0\geq 2, \\
    \left \{A \in \binom{N}{r} \colon A \leq \{a_1,\dots, a_r\} \right \} & \text{if}~a_0=1, a_1\geq 3,\\
    \left \{ \{i\} \cup A \colon A \in \binom{N}{r-1},~i\in \{1,2\},~A\leq \{a_2,\dots, a_r\},~i\notin A  \right \} & \text{otherwise.}
\end{cases}$$
The last case shows that in general the shadow of an initial segment is not an initial segment.
For example, the shadow of the initial segment $\{A\in \binom{N}{4}\colon A \leq \{1,2,4,5\}\}$ contains $\{2,4,5\}$
but not $\{1,n-1,n\}$ whenever $n\geq 6$; the latter being lexicographically smaller than the former.
For this reason we proceed to augment the shadow of an initial segment by adding the least number of sets in order to make it an initial segment.
Concretely, we define
$$\overline{\partial I(N,r,m)} = 
\begin{cases} 
    \binom{N}{r} & \text{if}~a_0\geq 2, \\
    \left \{A \in \binom{N}{r} \colon A \leq \{a_1,\dots, a_r\} \right \} &  \text{otherwise.}
\end{cases}$$
For an integer $m$ we set $\partial_r(n,m) = |\overline{\partial I(N,r,m)}|$.

Now, we proceed to define the lexicographically smallest simplicial complex with a given $f$-vector.
Let $\fvec = (f_0,f_1,\dots)$ be a non-negative integer vector, set $n=f_0$ and $K_{\fvec} = \bigcup_{r\geq 0}I(N,r,f_r)$.
It is immediate from this definition that the $f$-vector $f(K_{\fvec})$ of the set system $K_{\fvec}$ is $f(K_{\fvec}) = \fvec$.
The following lemma gives the necessary and sufficient condition for $K_{\fvec}$ to be a simplicial complex.
\begin{lemma}
    Let $\fvec$ be a non-negative integer vector. Then, $K_{\fvec}$ is a simplicial complex if and only if for every $r$, $\partial_r(n,f_r) \leq f_{r-1}$, where $n=f_0$.    
  \end{lemma}
\begin{proof}
  If $K_{\fvec}$ is a simplicial complex, let $\{a_0,\dots, a_r\}$ be the last element in  $(K_{\fvec})_r$.   
  If $\partial (K_{\fvec})_r$ is an initial segment then it coincides with the augmented boundary and consequently its size is $\partial_r(n,f_r)$.
  Now the claim follows since $$\partial_r(n,f_r) = |\overline{\partial (K_{\fvec})_r}| = |\partial (K_{\fvec})_r|\leq |(K_{\fvec})_{r-1}| = f_{r-1}$$
  where the inequality follows from the assumption that $K_{\fvec}$ is a simplicial complex.
  If $\partial (K_{\fvec})_r$ is not an initial segment then in particular $a_0=1$. Since $\{a_1,\dots,a_r\} \in (K_{\fvec})_{r-1}$ and the latter is an initial segment the inequality follows as well.

    In order to show the reverse implication, since $\partial (K_{\fvec})_r \subseteq \overline{\partial (K_{\fvec})_r}$ it is enough to verify that $\overline{\partial (K_{\fvec})_r}$ is contained in $(K_{\fvec})_{r-1}$. This follows from the fact that both are initial segments and that $$|\overline{\partial (K_{\fvec})_r}| = \partial_r (n,f_r) \leq f_{r-1} = |(K_{\fvec})_{r-1}|.$$
  \end{proof}
  
We say that a simplicial complex $K$ is a \emph{lex-segment complex} if $K = K_{\fvec(K)}$.

\begin{remark}
    If instead the reverse lexicographical order is used, namely $A\prec B$ if $\max (A\cup B)\setminus (A\cap B)\in B$, 
    then it 
    is enough to use the shadow operator without the augmentation step, since in this case the shadow of an initial segment is always an initial segment. 
    These are known as \emph{compressed} simplicial complexes.
    However, in general the algebraic shifting with respect to the reverse lexicographical order does not preserve the Betti numbers.
\end{remark}

We now modify this construction to incorporate Betti numbers, a l\'a  Bj\"{o}rner-Kalai~\cite{Bjorner-Kalai-ACTA}.
For it, let $\fvec=(f_0,f_1,\dots)$ and $\bbeta=(\beta_0,\beta_1,\dots)$ be two non-negative integer vectors and set $n=f_0$.
Similarly to~\cite{Bjorner-Kalai-ACTA} for $r\geq 0$ we set $\chi_{r-1}=\sum_{j\geq r}(-1)^{j-r}(f_j-\beta_j)$
and consider the families of sets $E = \bigcup_{r\geq 0} I(N\setminus \{1\},r,\chi_r+\beta_r)$ and
$C=\bigcup_{r\geq 0}I(N\setminus \{1\},r,\chi_r)$ and finally set $K_{\fvec,\bbeta} = (1*C)\cup E$
to be the cone over $C$ with apex $1$, union $E$. 

\begin{lemma}
\label{l:homlexsegment}
Let $\fvec,~\bbeta$ be two non-negative integer vectors.
Then, $K_{\fvec, \bbeta}$ is a simplicial complex with $f(K_{\fvec,\bbeta}) = \fvec$ and $\beta (K_{\fvec,\bbeta}) = \bbeta$
if and only if $\chi_{-1}=1$ and $\partial_r (n-1,\chi_r + \beta_r) \leq \chi_{r-1}$ for $r\geq 1$, where $n=f_0$. 
\end{lemma}
\begin{proof}
Assume the numerical conditions hold.
  First we show that $K_{\fvec,\bbeta}$ is a simplicial complex by verifying that $\partial (K_{\fvec,\bbeta})_r \subseteq (K_{\fvec,\bbeta})_{r-1}$.
  Since $(K_{\fvec,\bbeta})_r = \{ \{1\} \cup A \colon A\in C_{r-1}\}\cup E_r$ then
  $\partial (K_{\fvec,\bbeta})_r = \{\{1\} \cup A \colon A\in \partial C_{r-1}\} \bigcup \partial E_r \bigcup C_{r-1}$.
  From their definitions it follows that $C_{r-1}\subseteq E_{r-1}$. 
  Because $$|\partial C_{r-1}|\leq |\overline{\partial C_{r-1}}| = \partial_{r-1}(n-1,\chi_{r-1}) \leq \partial_{r-1}(n-1,\chi_{r-1}+\beta_{r-1})\leq \chi_{r-2} = |C_{r-2}|$$ 
  then $\partial C_{r-1}\subseteq \overline{\partial C_{r-1}}\subseteq C_{r-2}$ since the last two are initial segments.
  Similarly, since $|\partial E_r|\leq |\overline{\partial E_r}|=\partial_r(n-1,\chi_r+\beta_r) \leq \chi_{r-1}= |C_{r-1}|$ 
  then $\partial E_r \subseteq \overline{\partial E_r}\subseteq C_{r-1}$ since the last two are initial segments. From these three containments we conclude that $\partial (K_{\fvec,\bbeta})_r \subseteq (K_{\fvec,\bbeta})_{r-1}$.

    A simple computation shows that $f_r(K_{\fvec,\bbeta})=f_{r-1}(C)+f_r(E)=\chi_{r-1}+\chi_r+\beta_r=f_r.$    
    By the first paragraph we have that $K_{\fvec,\bbeta}$ is a near-cone and consequently its reduced Betti numbers are given by~\cite[Theorem 4.3]{Bjorner-Kalai-ACTA}
    \[
    \beta_r(K_{\fvec,\bbeta}) =|\{B\in K_{\fvec,\bbeta}\colon |B|=r+1, \{1\} \cup B \notin K_{\fvec,\bbeta}\}| = 
    f_r(E\setminus C) = 
    (\chi_r+\beta_r)-\chi_r=\beta_r.
    \]
    
    To verify the reverse implication we observe that since the cone $1*C$ has trivial homology, any nontrivial homology representative of $K_{\fvec,\bbeta}$ must contain at least one element of $E\setminus C$ in its support.
    Because $f_r(E\setminus C)=\beta_r(K_{\fvec,\bbeta})$, there exists a homology basis each having exactly one element of $E\setminus C$ in their support; as can be seen by Gauss elimination.
    Now, if $\partial E_r$ is not contained in $C_{r-1}$ it introduces a linear dependency among homology basis elements indexed by $E_{r-1}\setminus C_{r-1}$, which is a contradiction. Since $C_{r-1}$ is an initial segment, if $\partial E_r \subseteq C_{r-1}$ then indeed $\overline{\partial E_r} \subseteq C_{r-1}$ and consequently $\partial_r(n-1,\chi_r+\beta_r)\leq \chi_{r-1}$ follows by comparing their cardinalities.
\end{proof}
We say that a simplicial complex $K$ is a \emph{homology lex-segment complex} if $K=K_{f(K),\beta(K)}$.

\noindent
{\bf One-dimensional homology lex-segment complexes.}
Let $\fvec=(n,f_1)$ and $\bbeta=(\beta_0,\beta_1)$ such that $\chi_{-1} = 1$ and $\partial_1(\chi_1+\beta_1)\leq \chi_0$.
Then, $C_0=[2,n-\beta_0]$ while $(E\setminus C)_0=[n-\beta_0+1,n]$.
Since $f_2=\beta_2=0$ then $\chi_1=0$ and consequently $C_1=\emptyset$. 
On the other hand, $(E\setminus C)_1 = I(N\setminus \{1\},1,\beta_1) = \{A \in \binom{[2,n-\beta_0]}{2} \colon A \leq \{c,d\}\}$ since $\partial E_1 \subseteq C_0$.
If $n \geq 2 + \beta_0 + \beta_1$ the final expression can be simplified as
\begin{equation}
  \label{eq:hseggraph}
  K_{\fvec,\bbeta} =  [n] \bigcup \left \{ A \in \binom{[n-\beta_0]}{2} \colon A \leq \{2,\beta_1+2\} \right \}.
\end{equation}
Then $K_{\fvec,\bbeta}$ is indeed a homology lex-segment complex.

\noindent
{\bf Surface triangulation and homology lex-segment complexes.}
Let $S_g$ be an closed connected orientable surface of genus $g$ and let $\fvec,\bbeta$ be the face and reduced Betti vectors of a triangulation of $S_g$. 
In particular, these only depend on the number of vertices and the genus $g$, i.e., $\fvec=(n,3n+6(g-1),2n+4(g-1))$ and $\bbeta=(0,2g,1)$.
First, since $\chi_0=n-1$ we have that $C_0=N\setminus \{1\}$ and consequently $\{1,n\}\in (1*C)_1$.
Similarly to the 
one dimensional case,
let us assume that $n\geq 6g+4$.
Then, since $\chi_1 = 2n+4g-5$ we have that $C_1 = \{A \in \binom{N\setminus \{1\}}{2} \colon A \leq \{4,4+4g\}\}$ and so given that $\chi_2=0$ we can conclude that $(1*C)$ has facets $\{A \in \binom{N}{3} \colon A \leq \{1,4,4+4g\}\}$.

Since $\beta_1=2g$ we have that $(E\setminus C)_1=\{A \in \binom{N\setminus \{1\}}{2}\colon \{4,4+4g\} < A \leq \{4,4+6g\}\}$. Finally, given that $\beta_2=1$ implies that $(E\setminus C)_2=\{2,3,4\}$. We can conclude that
\begin{equation}
\label{eq:hsegsrf}
K_{\fvec,\bbeta} = [n] \bigcup \left \{A \in \binom{[n]}{2}\colon A \leq \{4,4+6g\} \right \} \bigcup \left \{ A \in \binom{[n]}{3} \colon A \leq \{1,4,4+4g\} \right \} \bigcup \{\{2,3,4\}\}.
\end{equation}

Let $N_g$ be a non-orientable surface or genus $g$ (without boundary) and let $\fvec,\bbeta$ be the face and reduced Betti vectors of a triangulation of $N_g$.
Similarly as above let us assume that $f_0=n\geq 4+3g$.
The homology lex-segment will depend on the characteristic of the field. For a field of characteristic $0$ we have that
$$K_{\fvec,\bbeta}:= [n] \bigcup \left \{A \in \binom{[n]}{2}\colon A \leq \{4,4+3g\} \right \} \bigcup \left \{ A \in \binom{[n]}{3} \colon A \leq \{1,4,5+2g\} \right \}.$$
On the other hand, for a field of characteristic two
$$K_{\fvec,\bbeta}:= [n] \bigcup \left \{A\in \binom{[n]}{2}\colon A \leq \{4,4+3g\} \right \} \bigcup \left \{ A \in \binom{[n]}{3} \colon A \leq \{1,4,4+2g\} \right \} \bigcup \{\{2,3,4\}\}.$$
Then $K_{\fvec,\bbeta}$ is indeed a homology lex-segment complex in all of the cases above.

\end{document}